\documentclass[11pt]{amsart}

\usepackage{amsmath}
\usepackage{amssymb}
\usepackage{amsthm}
\usepackage{amscd}
\usepackage{mathrsfs}

\newcommand{\ri}{\mathfrak{o}}
\newcommand{\mi}{\mathfrak{p}}

\newcommand{\Sch}[1]{\mathcal{C}_c^\infty({#1})}
\newcommand{\Z}{\mathbf{Z}}
\newcommand{\C}{\mathbf{C}}
\newcommand{\Q}{\mathbf{Q}}
\newcommand{\F}{\mathbf{F}}

\newcommand{\fL}{\mathscr{L}(s, \pi, \land^2)}
\newcommand{\fLl}{\mathscr{L}(s_2, \pi, \land^2)}
\newcommand{\LJS}{{L}_{JS}(s, \pi, \land^2)}
\newcommand{\exL}{{L}(s, \pi, \land^2)}
\newcommand{\lra}{\longrightarrow}
\newcommand{\ds}{\displaystyle}
\newcommand{\ethm}{\end{theorem}}

\newcommand{\p}{\varpi}

\def\Section#1{\section{#1}\setcounter{equation}{0}}

\theoremstyle{plain}
\newtheorem{thm}[equation]{Theorem}
\newtheorem{lem}[equation]{Lemma}
\newtheorem{prop}[equation]{Proposition}

\theoremstyle{definition}
\newtheorem{defn}[equation]{Definition}

\newtheorem{rem}[equation]{Remark}

\title{Local newforms and formal exterior square $L$-functions}
\author{Michitaka Miyauchi and Takuya Yamauchi}
\date{\today}
\address{
Department of Mathematics, Faculty of Science\\
Kyoto University\\
Oiwake Kita-Shirakawa Sakyo Kyoto 606-8502 JAPAN
}
\email{miyauchi@math.kyoto-u.ac.jp}

\address{
Department of mathematics, Faculty of Education\\
Kagoshima University\\
Korimoto 1-20-6 Kagoshima 890-0065 JAPAN}
\email{yamauchi@edu.kagoshima-u.ac.jp}

\begin{document}
\begin{abstract}
Let
$F$ be a non-archimedean local field of characteristic zero.
Jacquet and Shalika attached a family of zeta integrals 
to unitary irreducible generic representations $\pi$ of 
$\mathrm{GL}_n(F)$.
In this paper, we show that 
Jacquet-Shalika integral attains a certain $L$-function,
so called the formal exterior square $L$-function,
when the Whittaker function is associated to a newform for $\pi$.
By consideration on the Galois side,
formal exterior square $L$-functions
are equal to exterior square $L$-functions
for some principal series representations.
\end{abstract}

\maketitle
\pagestyle{myheadings}
\markboth{}{}

\section{Introduction}
Let $F$ be a non-archimedean local field of characteristic zero
and
$\ri$ its ring of integers with the maximal ideal $\mi$.
Let
$\pi$ be an irreducible admissible representation 
of $\mathrm{GL}_n(F)$.
Via the local Langlands correspondence,
there exists a Weil-Deligne representation $\rho$
of the Weil group $W_F$
associated to $\pi$.
The exterior square $L$-function of $\pi$
is defined by 
\[
\exL = L(s, \land^2 \rho),
\]
where $L(s, \land^2 \rho)$ is the $L$-factor of 
the representation $\land^2 \rho$ of $W_F$.
We suppose that $\pi$ is unitary and generic.
We denote by $\mathcal{W}(\pi, \psi)$
the Whittaker model of $\pi$,
and by $\Sch{F^{m}}$
the space of Schwartz functions on $F^{m}$.
To give an integral representation of $\exL$,
Jacquet and Shalika in \cite{J-S} 
introduced a family of zeta integrals
of the form $J(s, W, \Phi)$ for $n$ even,
and $J(s, W)$ for $n$ odd,
where $W\in \mathcal{W}(\pi, \psi)$ and $\Phi \in \Sch{F^{n/2}}$.
In {\it loc. cit.},
they showed that
the integral $J(s, W, \Phi)$ attains $\exL$
when $\pi$ is unramified and $W$ is spherical.
The key for
unramified computation is 
the explicit formula for the spherical Whittaker functions
given by Casselman-Shalika \cite{CS} and Shintani \cite{Shintani}.

It is natural to ask about ramified representations.
Jacquet, Piatetski-Shapiro and Shalika introduced 
the concept of newforms 
for generic representations of $\mathrm{GL}_n(F)$
in \cite{JPSS},
which is an extension of that of spherical vectors for unramified representations.
Recently,
Matringe \cite{Matringe}
and the first author \cite{M5} independently
gave an explicit formula for Whittaker functions associated to 
newforms on the diagonal torus.
We apply this formula to
compute the integral $J(s, W, \Phi)$
when $W$ is associated to a newform.

To state our results,
we introduce the notion of  {\it the formal exterior square $L$-functions}.
For 
an irreducible admissible representation $\pi$ of $\mathrm{GL}_n(F)$,
its standard $L$-function can be written as
\[
L(s, \pi) =\prod_{i = 1}^n (1-\alpha_i q^{-s})^{-1},\ \alpha_i \in \C,
\]
where $q$ denotes the cardinality of the residue field of $F$.
We define the formal exterior square $L$-function of $\pi$ 
by
\[
\fL
= \prod_{1 \leq i < j \leq n} (1-\alpha_i \alpha_j q^{-s})^{-1}.
\]
It is known that $\fL$ is equal to $\exL$
for unramified principal series representations,
and one may check that 
$\fL$ divides $\exL$ in general (Theorem~\ref{thm:galois} (ii)).
In this paper,
we shall show the following:
\begin{thm}\label{thm:intro}
Let $\pi$ be a unitary irreducible generic representation of
$\mathrm{GL}_n(F)$.
Suppose that a function $W$ in $\mathcal{W}(\pi, \psi)$
is associated to a newform for $\pi$.
Then the integral $J(s, W, \Phi_{c})$ ($J(s, W)$ if $n$ is odd)
is a constant multiple of $\fL$,
where $c$ is the conductor of $\pi$
and $\Phi_{c}$ is the characteristic function 
of $\mi^c \oplus \cdots \oplus \mi^c \oplus (1+\mi^c) \subset 
F^{n/2}$.
\end{thm}

Theorem~\ref{thm:intro} has several applications.
We summarize them comparing the recent progress in this topic.
The Jacquet-Shalika integrals attached to a unitary irreducible 
generic representation $\pi$
span a fractional ideal $I_\pi$ of $\C[q^{-s}, q^{s}]$.
It is an important fact that 
$I_\pi$ contains $1$.
Due to this, we may define Jacquet-Shalika's exterior square 
$L$-function $\LJS$
to be the normalized generator of $I_\pi$.
Kewat and Raghunathan in \cite{K-R} have already
mentioned that this is implicitly proved by Belt in \cite{Belt}.
Theorem~\ref{thm:intro} gives an alternative (and brief) proof of this 
because it implies that $\fL$ is contained in $I_\pi$.
Additionally,
Theorem~\ref{thm:intro} says that 
$\fL$ divides $\LJS$.
Thus the poles of $\fL$ are also those of $\LJS$ (Theorem~\ref{thm:zeros}).
Recently,
Kewat and Raghunathan in \cite{K-R}
showed the coincidence of 
$\LJS$ and $\exL$
for all the essentially square integrable representations of $\mathrm{GL}_n(F)$,
and for all the generic representations when $n$ is even.
Although Theorem~\ref{thm:zeros}
is obvious for such representations via arguments on the Galois side
(see Theorem~\ref{thm:galois} (ii)),
it provides an evidence of the equality of 
$\LJS$ and $\exL$ for the odd case.

It is still an open problem to find 
Whittaker functions 
which attain exterior square $L$-functions through Jacquet-Shalika integral.
We give an example of 
some principal series representations $\pi$
for which $\fL$ equals to $\exL$ (Proposition~\ref{prop:H}).
Therefore
Whittaker newforms attain $\exL$
for such representations.

On the other hand,
there is
another kind of zeta integrals
related to exterior square $L$-functions,
introduced by Bump and Friedberg \cite{B-F}.
For an irreducible generic representation $\pi$
of $\mathrm{GL}_n(F)$,
Bump-Friedberg integral has the form
$Z(s_1, s_2, W, \Phi)$,
where $W \in \mathcal{W}(\pi, \psi)$
and $\Phi \in \Sch{F^{\lfloor (n+1)/2\rfloor}}$.
For Bump-Friedberg integral,
we obtain the following
\begin{thm}\label{thm:intro2}
Let $\pi$ be an irreducible generic representation of
$\mathrm{GL}_n(F)$.
Suppose that a function $W$ in $\mathcal{W}(\pi, \psi)$
is associated to a newform for $\pi$.
Then the integral $Z(s_1, s_2, W, \Phi_c)$
is a constant multiple of $L(s_1, \pi)\fLl$,
where $c$ is the conductor of $\pi$
and $\Phi_{c}$ is the characteristic function 
of $\mi^c \oplus \cdots \oplus \mi^c \oplus (1+\mi^c) \subset 
F^{\lfloor (n+1)/2\rfloor}$.
\end{thm}

This paper is organized as follows.
In section~\ref{section:pre},
we define the formal exterior square $L$-functions,
and relate them with newforms.
We show that 
Jacquet-Shalika integrals attain the formal exterior square $L$-functions when Whittaker functions are associated to newforms
for the even case
in section~\ref{sec:even},
and for the odd case in section~\ref{sec:odd}.
We consider Bump-Friedberg integral in section~\ref{sec:BF}.
The meaning of the formal exterior square $L$-functions 
on the Galois side is given in section~\ref{sec:append}.

\medskip
\noindent
{\bf Acknowledgements} \
The first author expresses his appreciation 
to Taku Ishii for giving an 
introductory talk on this subject at Tambara Institute of Mathematical Sciences The University of Tokyo,
and to Takayuki Oda for inviting him to that conference.
The authors would like to thank 
Tomonori Moriyama for helpful conversations
and Yoshi-hiro Ishikawa for his useful comments.
The first author is partially supported by JSPS Grant-in-Aid for Scientific Research No.21540017.
The second author is partially supported by JSPS Grant-in-Aid for Scientific Research No.23740027. 

\Section{Preliminaries}\label{section:pre}
In this section,
after fixing notations,
we define formal exterior square $L$-functions for
irreducible generic representations of $\mathrm{GL}(n)$,
and relate them with newforms.
\subsection{Notation}
Let 
$F$ be a non-archimedean local field of characteristic zero,
$\ri$ its ring of integers,
$\mi$ the maximal ideal in $\ri$,
and
$\p$ a generator of $\mi$.
Let $\nu$ denote the valuation on $F$ normalized so that $\nu(\p) = 1$.
We write $|\cdot|$  for the absolute value of $F$
normalized so that $|\p| = q^{-1}$,
where 
$q$ stands for the cardinality of the residue field $\ri/\mi$
of $F$.
Throughout this paper,
we fix a non-trivial additive character $\psi$ of $F$
whose conductor is $\ri$,
that is,
$\psi$ is trivial on $\ri$ and not trivial on $\mi^{-1}$.

We 
set
$G_n = \mathrm{GL}_n(F)$.
Let
$B_n$ denote the Borel subgroup of $G_n$ consisting of the upper triangular 
matrices,
$T_n$ the diagonal torus in $G_n$
and $U_n$ the unipotent radical of $B_n$.
We write $\delta_{B_n}$ for the modulus character of $B_n$.
We define a subgroup $T_{n,1}$ of $T_n$
by
\[
T_{n,1}
=\{\mathrm{diag}(a_1, a_2, \ldots, a_{n-1}, 1)\, |\,
a_1, \ldots, a_{n-1} \in F^\times\}.
\]
We use the same letter 
$\psi$ for the following character of $U_n$ induced from $\psi$:
\[
\psi(u) = \psi(\sum_{i=1}^{n-1}u_{i,i+1}),\
\mathrm{for}\ u = (u_{ij}) \in U_n.
\]
For 
an irreducible generic representation $(\pi, V)$ of $G_n$,
we denote by
$\mathcal{W}(\pi, \psi)$ its  Whittaker model
with respect to $\psi$.

\subsection{Formal exterior square $L$-functions}

Let $\pi$ be an irreducible generic representation of
$G_n$.
We denote by $L(s, \pi)$ the $L$-factor of $\pi$
defined in \cite{GJ}.
Since the degree of $L(s, \pi)$ is equal to or
less than $n$,
we can write $L(s, \pi)$ as
\begin{eqnarray}\label{eq:L}
L(s, \pi) =\prod_{i = 1}^n (1-\alpha_i q^{-s})^{-1},\ \alpha_i \in \C.
\end{eqnarray}
Here we allow the possibility that
$\alpha_i = 0$.

We define {\it the formal exterior square $L$-function} of $\pi$
by
\[
\fL
= \prod_{1 \leq i < j \leq n} (1-\alpha_i \alpha_j q^{-s})^{-1}.
\]
We say that $\pi$ is {\it unramified}
if $\pi$ has a non-zero $\mathrm{GL}_n(\ri)$-fixed vector.
Suppose that $\pi$ is unramified.
Then $\fL$ coincides with the exterior square $L$-function 
$L(s, \pi, \land^2)$ of 
$\pi$ defined through the local Langlands correspondence (\cite{J-S}).

\begin{prop}\label{prop:esi}
Suppose that $\pi$ is an
irreducible, essentially square integrable representation of $G_n$.
Then we have
\[
\fL = 1.
\]
\end{prop}
\begin{proof}
If $\pi$ is an
irreducible, essentially square integrable representation of $G_n$,
then 
the degree of 
$L(s, \pi)$ is equal to or less than 1 (see \cite{Jacquet2}).
The assertion follows immediately from this. 
\end{proof}
Let
$X = (X_1, \ldots, X_{n})$ be an $n$-tuple of  indeterminates.
If $f \in \Z^{n}$ 
satisfies $f_1 \geq \ldots \geq f_n \geq 0$, then
we denote by $s_f(X)$ the Schur polynomial 
in $X_1, \ldots, X_{n}$
associated to $f$,
that is,
\begin{eqnarray*}
s_f(X) = \frac{|(X_j^{f_i+n-i})_{1 \leq i, j\leq n}|}{\prod_{1 \leq i < j \leq n}(X_i-X_j)}
\end{eqnarray*}
(see \cite{Mac} Chapter I, section 3).
Since $s_f(X)$  is a symmetric polynomial,
the number $s_f(\alpha) = s_f(\alpha_1, \ldots, \alpha_n)$
is independent of the ordering of $\alpha_1, \ldots, \alpha_n$
in (\ref{eq:L}).

Until the end of this subsection,
we assume that the degree $k$ of $L(s, \pi)$ is less than $n$.
Then
we can take $\alpha_1, \ldots, \alpha_n$
so that 
$\alpha_1 \alpha_2 \cdots \alpha_k \neq 0$
and $\alpha_{k+1} = \alpha_{k+2} = \cdots = \alpha_n = 0$.
Thus we have
\begin{eqnarray*}
\fL
= \prod_{1 \leq i < j \leq k} (1-\alpha_i \alpha_j q^{-s})^{-1}.
\end{eqnarray*}
\begin{lem}\label{lem:Schur}
With the above notations,
if $s_f(\alpha) \neq 0$,
then we have $f_{k+1} = f_{k+2} = \cdots = f_n = 0$
and
$s_f(\alpha)
= s_{(f_1, \ldots, f_k)}(\alpha_1, \ldots, \alpha_k)$.
\end{lem}
\begin{proof}
The assertion follows from 
the equations
\[
s_f(\alpha)
= (\alpha_1 \cdots \alpha_n)^{f_n} s_{(f_1-f_n, \ldots, f_{n-1}-f_n, 0)}(\alpha)
\]
and
\[
s_{(f_1, \ldots, f_{n-1}, 0)}(\alpha_1, \ldots, \alpha_{n-1}, 0)
=s_{(f_1, \ldots, f_{n-1})}(\alpha_1, \ldots, \alpha_{n-1}).
\]
\end{proof}

For $f = (f_1, \ldots, f_k) \in Z^k$
which satisfies
$f_1 \geq \ldots \geq f_k \geq 0$,
we denote by
$r(f) = r(f_1,\ldots, f_k)$
the irreducible representation of $\mathrm{GL}_k(\C)$
with dominant weight $f$.
Set $A = \mathrm{diag}(\alpha_1, \ldots, \alpha_k)$.
Then 
$\mathrm{tr}(r(f)A)$
equals to $s_f(\alpha_1, \ldots, \alpha_k)$ (\cite{Mac}).
By the results in \cite{J-S}
subsection 2.2,
for $k = 2h$,
we have
\begin{eqnarray*}
\fL
& = &
\sum_{l \geq 0} q^{-ls}
\sum_{f_1+ f_2+\cdots +f_h = l}
s_{(f_1, f_1, f_2, f_2, \ldots, f_h, f_h)}(\alpha_1, \ldots, \alpha_k),
\end{eqnarray*}
and
for $k = 2h+1$, we get
\begin{eqnarray*}
\fL
& = &
\sum_{l \geq 0} q^{-ls}
\sum_{f_1+ f_2+\cdots +f_h = l}
s_{(f_1, f_1, f_2, f_2, \ldots, f_h, f_h, 0)}(\alpha_1, \ldots, \alpha_k),
\end{eqnarray*}
where it is understood that the sum is taken for 
$f_1 \geq f_2 \geq \ldots \geq f_h \geq 0$.
By Lemma~\ref{lem:Schur},
one can observe that
if $n =2m$, then
\begin{eqnarray}\label{eq:fL_even}
\fL
& = &
\sum_{l \geq 0} q^{-ls}
\sum_{f_1+ \cdots +f_{m-1} = l}
s_{(f_1, f_1, \ldots, f_{m-1}, f_{m-1}, 0, 0)}(\alpha_1, \ldots, \alpha_n),
\end{eqnarray}
and if $n = 2m+1$,
then
\begin{eqnarray}\label{eq:fL_odd}
\fL
& = &
\sum_{l \geq 0} q^{-ls}
\sum_{f_1+ f_2+\cdots +f_m = l}
s_{(f_1, f_1, f_2, f_2, \ldots, f_m, f_m, 0)}(\alpha_1, \ldots, \alpha_n).
\end{eqnarray}

\subsection{Newforms}
Put $K_{n,0} = \mathrm{GL}_n(\ri)$.
For each positive integer $r$,
let $K_{n,r}$ be the subgroup of $K_{n,0}$
consisting of the elements $k = (k_{ij})$ in $K_{n,0}$
which satisfy
\[
(k_{n1}, k_{n2}, \ldots, k_{nn}) \equiv (0, 0, \ldots, 0, 1) \pmod{\mi^r}.
\]
For any integer $r$,
let $\Phi_r$ denote 
the characteristic function of 
$\mi^r \oplus \cdots \oplus \mi^r \oplus (1+\mi^r) \subset 
F^n$.
The following lemma determines the support of the function 
$g \in G_n \mapsto \Phi_r(e_n g)$,
where $e_n = (0, 0, \ldots, 0,1) \in F^n$.
\begin{lem}\label{lem:decomp}
Suppose that $r$ is positive.
Then we have
\[
\Phi_r(e_ng)
=
\left\{
\begin{array}{cl}
1, & \mbox{if}\ g \in U_nT_{n,1}K_{n,r};\\
0, & \mbox{otherwise},
\end{array}
\right.
\]
for $g \in G_n$
\end{lem}
\begin{proof}
Clearly,
we have $\Phi_r(e_ng) = 1$
for $g \in U_nT_{n,1}K_{n,r}$.
We shall prove the converse statement.
By the Iwasawa decomposition
$G_n = U_n T_n K_{n, 0}$,
we can write $g$ in $G_n$
as 
$g = utk$,
where
$u \in U_n$, $t \in T_n$, $k \in K_{n,0}$.
Since the function $g \mapsto \Phi_r(e_n g)$ is
left $U_n$-invariant,
we may assume $g = tk$.
We write $t = \mathrm{diag}(t_1, t_2, \ldots, t_n)$,
$t_i \in F^\times$.
Suppose that $\Phi_r(e_n tk) = 1$.
Then we obtain
$|t_n k_{ni}| \leq q^{-r}$, for $1 \leq i \leq n-1$
and $|t_n k_{nn}| = 1$.
This implies $|k_{ni}| < |k_{nn}|$ for $1 \leq i \leq n-1$.
Since $k$ lies in $K_{n,0}$,
we have $|k_{nj}| \leq 1$ for all $1 \leq j \leq n$,
and there exists at least one $j$ such that $|k_{nj}| = 1$.
Thus we get $|k_{nn}| = 1$,
and hence $|t_n| = 1$.
So we may assume that $t$ lies in $T_{n,1}$.
In this case,
the equation $\Phi_r(e_ntk) = \Phi_r(e_nk) = 1$
precisely
means that $k$ belongs to $K_{n, r}$.
This completes the proof.
\end{proof}

Let $\pi$ be an irreducible generic representation of $G_n$.
We write $V(r)$ for the space of $K_{n,r}$-fixed vectors in $V$.
Due to \cite{JPSS} (5.1) Th\'{e}or\`{e}me (ii),
there exists a non-negative integer 
$r$ such that
$V(r)\neq \{0\}$.
We denote by $c(\pi)$ the smallest integer with this property.
We 
call $c(\pi)$
{\it the conductor of $\pi$},
and $V(c(\pi))$ {\it the space of newforms for $\pi$}.
By \cite{JPSS} (5.1) Th\'{e}or\`{e}me (ii) again,
we have
\begin{eqnarray}\label{eq:mult_one}
\dim V(c(\pi)) = 1.
\end{eqnarray}

For simplicity,
we say that an element $W$ in $\mathcal{W}(\pi, \psi)$
is a newform 
if $W$ is the Whittaker function associated to a newform for $\pi$.
It follows from 
\cite{M5} Theorem 4.1
and \cite{Shintani}
that
a newform $W$ in $\mathcal{W}(\pi, \psi)$
is determined by its value at $1 \in G_n$.
\begin{prop}\label{prop:nonzero}
Let 
$W$ be a newform in $\mathcal{W}(\pi, \psi)$.
If an element $t = \mathrm{diag}(t_1, \ldots, t_n) \in T_n$
satisfies
$W(t) \neq 0$,
then we have
$|t_i| \leq |t_{i+1}|$ for all $1 \leq i \leq n-1$.
\end{prop}
\begin{proof}
By using the central character of $\pi$,
we may assume that $t_n = 1$.
Set $f_i = \nu(t_i)$, for $1 \leq i \leq n-1$.
Since $W$ is $K_{n, c(\pi)}$-invariant,
we have
$W(t) = W(\mathrm{diag}(\p^{f_1}, \ldots, \p^{f_{n-1}},1))$.
Hence the proposition follows from \cite{M5} Proposition 1.2.
\end{proof}

We shall give an integral representation of formal exterior square $L$-functions.
We normalize Haar measures on $T_n$ and $T_{n,1}$
so that
the volumes of $T_n\cap K_{n,0}$ and of $T_{n,1}\cap K_{n,0}$
are one respectively.
Note that if the conductor $c(\pi)$ of $\pi$ is positive,
then the degree of $L(s, \pi)$ is less than $n$ (\cite{Jacquet2}).
\begin{prop}\label{prop:formal}
Let $\pi$ be an irreducible generic representation of $G_n$
whose conductor is positive
and let $W$ be the newform in $\mathcal{W}(\pi, \psi)$
such that $W(1) = 1$.

(i)
Suppose that $n =2m$.
For any element $a = \mathrm{diag}(a_1, \ldots, a_m) \in T_m$,
we set 
\[
b = \mathrm{diag}(a_1, a_1, a_2, a_2, \ldots, 
a_m, a_m) \in T_n.
\]
Then we have
\begin{eqnarray*}
\fL
& =& 
\int_{T_{m, 1}}
W(b )
\delta_{B_n}(b)^{-1/2} |\det a|^s  da.
\end{eqnarray*}

(ii)
Suppose that $n =2m+1$.
For each element $a = \mathrm{diag}(a_1, \ldots, a_m) \in T_m$,
we set $b = \mathrm{diag}(a_1, a_1, a_2, a_2, \ldots, 
a_m, a_m, 1) \in T_n$.
Then we have
\begin{eqnarray*}
\fL
& =& 
\int_{T_{m}}
W(b )
\delta_{B_n}(b)^{-1/2} |\det a|^s  da.
\end{eqnarray*}
\end{prop}
\begin{proof}
(i)
For $a = \mathrm{diag}(a_1, \ldots, a_{m-1}, 1) \in T_{m,1}$,
we put $f_i = \nu(a_i)$, for  $1 \leq i \leq m-1$.
Since $W$ is fixed by $K_{n, c(\pi)}$,
we have
\[
W(b) 
= W(\mathrm{diag}(\p^{f_1}, \p^{f_1}, \p^{f_2}, \p^{f_2},\ldots, \p^{f_{m-1}}, \p^{f_{m-1}}, 1, 1)).
\]
It follows from \cite{M5} Theorem 4.1
that
\[
W(b) = 
\left\{
\begin{array}{cl}
\delta_{B_n}^{1/2}(b)
s_{(f_1, f_1, \ldots, f_{m-1}, f_{m-1}, 0, 0)}(\alpha_1, \ldots, \alpha_n), & \mbox{if}\
f_1 \geq \ldots \geq f_{m-1} \geq 0;\\
0, & \mbox{otherwise}.
\end{array}
\right.
\]
Thus, (\ref{eq:fL_even}) implies the assertion.
Part (ii) follows from (\ref{eq:fL_odd})
in a similar fashion.
\end{proof}

\Section{Jacquet-Shalika integral: the even case}\label{sec:even}
We shall prove that Jacquet-Shalika integral 
attains the formal exterior square $L$-function
when the Whittaker function is associated to a newform.
In this section,
we consider the case when $n = 2m$.

Let $\pi$ be a unitary irreducible generic representation of
$G_n$.
In \cite{J-S},
Jacquet and Shalika introduced a family of 
zeta integrals which have the form $J(s, W, \Phi)$, $W \in \mathcal{W}(\pi, \psi)$,
$\Phi \in \Sch{F^m}$,
where $\Sch{F^m}$ denotes the space of locally constant,
compactly supported functions on $F^m$:
\begin{eqnarray*}
& J(s, W, \Phi)
= \displaystyle
\int_{U_m\backslash G_m} \int_{V_m\backslash M_m} W\left(
\sigma \left(
\begin{array}{cc}
1_m & Z\\
0 & 1_m
\end{array}
\right)
\left(
\begin{array}{cc}
g & 0\\
0 & g
\end{array}
\right)
\right)
\psi(-\mathrm{tr}{Z}) dZ
\Phi(e_m g) |\det g|^s  dg,
\end{eqnarray*}
where 
$M_m = M_m(F)$,
$V_m$ is the space of upper triangular matrices in $M_m$,
$e_m = (0, 0, \ldots, 0, 1) \in F^m$
and
$\sigma$ is the permutation of degree $n = 2m$
given by
\begin{eqnarray*}
& 
\sigma = 
\left(
\begin{array}{cccc|cccc}
1 & 2 & \ldots & m & m+1 & m+2 & \ldots & 2m\\
1 & 3 & \ldots & 2m-1 & 2 & 4 & \ldots & 2m
\end{array}
\right).
\end{eqnarray*}
By Proposition 1 in \cite{J-S} section 7
and \cite{Belt} Proposition 4.3,
there exists $\eta > 0$
such that the integral $J(s, W, \Phi)$
absolutely converges to a rational function in $q^{-s}$
for $\mathrm{Re}(s) > 1-\eta$.

We take Haar measure on $V_m\backslash M_m$
so that the volume of $V_m\backslash (V_m +M_m(\ri))$ is one.
By using the Iwasawa decomposition
$G_m = U_m T_m K_{m, 0}$,
we can write an element $g$ in $G_m$
as $g = uak$, $u \in U_m$, $a \in T_m$, $k \in K_{m, 0}$.
Then Haar measure $dg$ on $U_m\backslash G_m$
is decomposed into
\[
\int_{U_m\backslash G_m}dg = \int_{T_m}\delta_{B_m}(a)^{-1} da\, \int_{K_{m,0}}dk.
\]
We normalize Haar measures on $T_m$ and $K_{m,0}$
so that 
the volumes of $T_m\cap K_{m, c(\pi)}$ and of $K_{m, c(\pi)}$
are one respectively.
Then the following holds:
\begin{thm}\label{thm:main_even}
Let $\pi$ be a unitary irreducible generic representation of $\mathrm{GL}_{2m}(F)$
and let
$W$ be the newform in $\mathcal{W}(\pi, \psi)$
such that $W(1) = 1$.
Then we have
\[
J(s, W, \Phi_{c(\pi)}) = \fL,
\]
where
$c(\pi)$ is the conductor of $\pi$
and
$\Phi_{c(\pi)}$ is the characteristic function of
$\mi^{c(\pi)} \oplus \cdots \oplus \mi^{c(\pi)} \oplus 
(1+\mi^{c(\pi)}) \subset F^m$.
\end{thm}
\begin{proof}
If $c(\pi)$ is zero,
then the theorem follows from Proposition 2 in \cite{J-S}
section 7.
We suppose that $c = c(\pi)$ 
is positive.
The proof is quite similar to that for unramified representations.
By Lemma~\ref{lem:decomp},
the map
$g \mapsto \Phi_c(e_mg)$
is the characteristic function on $U_mT_{m, 1}K_{m, c}$.
We note that if $k$ belongs to  $K_{n, c}$,
then
$\left(
\begin{array}{cc}
k & 0\\
0 & k
\end{array}
\right)$
lies in $K_{n, c}$
and fixes $W$.
Thus we obtain
\begin{eqnarray*}
J(s, W, \Phi_c)
& =& 
\int_{T_{m, 1}}\int_{V_m\backslash M_m} W\left(
\sigma \left(
\begin{array}{cc}
1_m & Z\\
0 & 1_m
\end{array}
\right)
\left(
\begin{array}{cc}
a & 0\\
0 & a
\end{array}
\right)
\right)
\psi(-\mathrm{tr}{Z}) dZ
\delta_{B_m}(a)^{-1} |\det a|^s  da.
\end{eqnarray*}
Note that
$\left(
\begin{array}{cc}
1_m & Z\\
0 & 1_m
\end{array}
\right)
\left(
\begin{array}{cc}
a & 0\\
0 & a
\end{array}
\right)
=
\left(
\begin{array}{cc}
a & 0\\
0 & a
\end{array}
\right)
\left(
\begin{array}{cc}
1_m & a^{-1}Za\\
0 & 1_m
\end{array}
\right)
$
and
$\mathrm{tr}{Z} = \mathrm{tr}({a^{-1}Za})$.
So we get 
\begin{eqnarray*}
J(s, W, \Phi_c)
& =& 
\int_{T_{m, 1}}\int_{V_m\backslash M_m} W\left(
\sigma 
\left(
\begin{array}{cc}
a & 0\\
0 & a
\end{array}
\right)
\left(
\begin{array}{cc}
1_m & Z\\
0 & 1_m
\end{array}
\right)
\right)
\psi(-\mathrm{tr}{Z}) dZ
\delta_{B_m}(a)^{-2} |\det a|^s  da\\
& =& 
\int_{T_{m, 1}}\int_{V_m\backslash M_m} W\left(
\sigma 
\left(
\begin{array}{cc}
a & 0\\
0 & a
\end{array}
\right)
\left(
\begin{array}{cc}
1_m & Z\\
0 & 1_m
\end{array}
\right)\sigma^{-1}
\right)
\psi(-\mathrm{tr}{Z}) dZ
\delta_{B_m}(a)^{-2} |\det a|^s  da.
\end{eqnarray*}
The second equality follows because
$\sigma$ belongs to $K_{n, c}$ and fixes $W$.
Set $b
=
\sigma 
\left(
\begin{array}{cc}
a & 0\\
0 & a
\end{array}
\right)\sigma^{-1}
$.
Then we have
$b = \mathrm{diag}(a_1, a_1, a_2, a_2, \ldots, a_{m-1}, a_{m-1}, 1, 1)$,
where $a = \mathrm{diag}(a_1, a_2, \ldots, a_{m-1}, 1) \in 
T_{m, 1}$.
By the Iwasawa decomposition $G_n
= U_n T_n K_{n,0}$,
we can write 
\[
\sigma
\left(
\begin{array}{cc}
1_m & Z\\
0 & 1_m
\end{array}
\right)\sigma^{-1}
= u_Z t_Z k_Z,
\]
for $Z \in M_m$,
where $u_Z \in U_n$, $t_Z \in T_n$ and $k_Z \in K_{n, 0}$.
Since the $n$-th row of $u_Z t_Z k_Z$
is $(0, 0, \ldots, 0, 1)$,
we can take $t_Z$ and $k_Z$
so that $t_Z \in T_{n, 1}$
and $k_Z$
has the $n$-th row 
$(0, 0, \ldots, 0, 1)$.
This implies that $k_Z$ lies in $K_{n, c}$.
Hence we obtain
\begin{eqnarray*}
J(s, W, \Phi_c)
& =& 
\int_{T_{m, 1}}\int_{V_m\backslash M_m} W\left(
b u_Z t_Z
\right)
\psi(-\mathrm{tr}{Z}) dZ
\delta_{B_m}(a)^{-2} |\det a|^s  da\\
& =& 
\int_{T_{m, 1}}\int_{V_m\backslash M_m}
\psi(b u_Z b^{-1})
 W\left(
b t_Z
\right)
\psi(-\mathrm{tr}{Z}) dZ
\delta_{B_m}(a)^{-2} |\det a|^s  da.
\end{eqnarray*}
We write 
$b = \mathrm{diag}(b_1, \ldots, b_n)$
and $t_Z = \mathrm{diag}(t_1, \ldots, t_n)$.
It follows from Proposition~\ref{prop:nonzero}
that if $W(b t_Z) \neq 0$,
then we have
$|b_i t_i| \leq |b_{i+1}t_{i+1}|$,
for $1 \leq i \leq n-1$.
So we obtain
$|t_i| \leq |t_{i+1}|$, for $i$ odd.
By Proposition 4 in \cite{J-S} section 5,
we have
$|t_i| \geq 1$ for $i$ odd,
and $|t_i| \leq 1$ otherwise.
Thus we get $|t_i| = 1$ for all $i$.
Proposition 5 in \cite{J-S} section 5
says that $Z$ lies in $V_m +M_m(\ri)$.
So we may take $u_Z = t_Z = 1$
 if $W(b t_Z) \neq 0$,
and obtain
\begin{eqnarray*}
J(s, W, \Phi_c)
& =& 
\int_{T_{m, 1}}
W(b )
\delta_{B_n}(b)^{-1/2} |\det a|^s  da
\end{eqnarray*}
because $\delta_{B_m}(a)^2
=\delta_{B_n}(b)^{1/2}$.
Now the assertion follows from Proposition~\ref{prop:formal}
(i).
\end{proof}

Let $I_\pi$
be the subspace of $\C(q^{-s})$ 
spanned by $J(s, W, \Phi)$,
where 
$W \in \mathcal{W}(\pi, \psi)$
and
$\Phi \in \Sch{F^m}$.
We shall give an alternative proof of a result by Belt.
\begin{prop}[\cite{Belt} Theorem 2.2, \cite{K-R}]\label{prop:frac}
With the notation as above,
$I_\pi$  is a fractional ideal of $\C[q^{-s}, q^{s}]$
which contains $1$.
\end{prop}
\begin{proof}
By \cite{Kewat} p.~158,
$I_\pi$ is a $\C[q^{-s}, q^{s}]$-module.
Due to \cite{Belt} Proposition 4.3,
there exists a polynomial $Q(X) \in \C[X]$
such that
$Q(q^{-s})I_\pi \subset \C[q^{-s}, q^{s}]$.
So $I_\pi$ is a fractional ideal of $\C[q^{-s}, q^s]$.
By Theorem~\ref{thm:main_even},
$\fL$ is contained in $I_\pi$,
so is $1$.
\end{proof}

By Proposition~\ref{prop:frac},
we can define Jacquet-Shalika's exterior square $L$-function
as follows:
\begin{defn}[\cite{K-R} Definition 3.4]
Due to Proposition~\ref{prop:frac},
there exists a polynomial $P(X) \in \C[X]$
such that
$P(0) = 1$ and
$I_\pi = (1/P(q^{-s}))$.
We define Jacquet-Shalika's exterior square $L$-function
by
\[
\LJS = \frac{1}{P(q^{-s})}.
\]
\end{defn}

We give an alternative proof of 
results on the non-vanishing of Jacquet-Shalika integral in 
\cite{Belt} and \cite{Kewat}
by using Proposition~\ref{prop:frac}.
\begin{prop}[\cite{Belt} Theorem 2.2, \cite{Kewat} Proposition 5.1]
Let $\pi$ be a unitary irreducible generic representation
of $\mathrm{GL}_{2m}(F)$.
For any $s_0 \in \C$,
there exist $W \in \mathcal{W}(\pi, \psi)$
and $\Phi \in \Sch{F^m}$ such that
$J(s_0, W, \Phi) \neq 0$.
\end{prop}
\begin{proof}
By Proposition~\ref{prop:frac},
there are $W_1, \ldots, W_k \in \mathcal{W}(\pi, \psi)$
and $\Phi_1, \ldots, \Phi_k \in \Sch{F^m}$
such that
\[
\sum_{i = 1}^k J(s, W_i, \Phi_i) = 1.
\]
Now the proposition is obvious.
\end{proof}

On the holomorphy of Jacquet-Shalika integral,
we obtain the following
\begin{prop}[cf. \cite{Belt} Theorem 2.2]
For an
irreducible, square integrable representation 
$\pi$ of $\mathrm{GL}_{2m}(F)$,
there exist $W \in \mathcal{W}(\pi, \psi)$
and $\Phi \in \Sch{F^m}$
such that
\[
J(s, W, \Phi) =1.
\]
\end{prop}
\begin{proof}
The proposition follows from
Proposition~\ref{prop:esi} and
 Theorem~\ref{thm:main_even}.
\end{proof}

Finally,
we state a result on poles of $\LJS$.
\begin{thm}\label{thm:zeros}
Let $\pi$ be a unitary irreducible generic representation of $\mathrm{GL}_{2m}(F)$.
(i) Suppose that $s_0 \in \C$ is a pole of $L(s, \pi)$
whose order is equal to or more than two.
Then $2s_0$ is a pole of $\LJS$.

(ii)
Suppose that $s_1$ and $s_2$ are two distinct poles of $L(s, \pi)$.
Then $s_1+s_2$ is a pole of $\LJS$.
\end{thm}
\begin{proof}
By Theorem~\ref{thm:main_even},
the formal exterior square $L$-function $\fL$ is contained 
in the set $\LJS \C[q^{-s}, q^s]$.
So the theorem follows from the definition of $\fL$.
\end{proof}

\Section{Jacquet-Shalika integral: the odd case}\label{sec:odd}
In this section,
we consider the case when $n = 2m+1$.
Let $\pi$ be a unitary irreducible generic representation of
$G_n$.
Then
Jacquet-Shalika integral for $\pi$ has the form
$J(s, W)$, $W \in \mathcal{W}(\pi, \psi)$:
\begin{eqnarray*}
& J(s, W)
= \displaystyle
\int_{U_m\backslash G_m} \int_{V_m\backslash M_m} W\left(
\sigma \left(
\begin{array}{ccc}
1_m & Z &0\\
0 & 1_m & 0\\
0 & 0 & 1
\end{array}
\right)
\left(
\begin{array}{ccc}
g & 0 & 0\\
0 & g & 0\\
0 & 0 & 1
\end{array}
\right)
\right)
\psi(-\mathrm{tr}{Z}) dZ
|\det g|^{s-1}  dg,
\end{eqnarray*}
where 
$\sigma$ is the permutation of degree $n = 2m+1$
given by
\begin{eqnarray*}
& 
\sigma = 
\left(
\begin{array}{cccc|cccc|c}
1 & 2 & \ldots & m & m+1 & m+2 & \ldots & 2m & 2m+1\\
1 & 3 & \ldots & 2m-1 & 2 & 4 & \ldots & 2m & 2m+1
\end{array}
\right).
\end{eqnarray*}
We normalize Haar measures on 
$U_m\backslash G_m$ and $V_m\backslash M_m$
so that the volumes of $K_{m,0}$ and of $V_m\backslash (V_m +M_m(\ri))$
are one respectively.
Similar results to those for the even case hold.
We shall be brief here.

\begin{thm}\label{thm:main_odd}
Let $\pi$ be a unitary irreducible generic representation of $\mathrm{GL}_{2m+1}(F)$
and let
$W$ be the newform in $\mathcal{W}(\pi, \psi)$
such that $W(1) = 1$.
Then we have
\[
J(s, W) = \fL.
\]
\end{thm}
\begin{proof}
Along the lines in the proof of Theorem~\ref{thm:main_even},
the theorem follows from
Proposition~\ref{prop:formal}
(ii).
\end{proof}

Let $I_\pi$
be the subspace of $\C(q^{-s})$ 
spanned by $J(s, W)$,
$W \in \mathcal{W}(\pi, \psi)$.
As in Proposition~\ref{prop:frac},
the set 
$I_\pi$  is a fractional ideal of $\C[q^{-s}, q^{s}]$
which contains $1$.
Thus we can define 
Jacquet-Shalika's exterior square $L$-function by
\[
\LJS = \frac{1}{P(q^{-s})},
\]
where 
$P(X)$ is a polynomial in $\C[X]$
such that
$P(0) = 1$ and
$I_\pi = (1/P(q^{-s}))$.

In the odd case,
Schwartz functions are not involved with Jacquet-Shalika integral.
So we get the following
\begin{prop}[cf. \cite{Belt} Theorem 6.1]
Let $\pi$ be a unitary irreducible generic representation
of $\mathrm{GL}_{2m+1}(F)$
Then
there exists $W \in \mathcal{W}(\pi, \psi)$
such that
$J(s, W) =1$.
\end{prop}
\begin{proof}
The proposition follows from the fact that 
$I_\pi$ contains $1$.
\end{proof}

We note that Theorem~\ref{thm:zeros}
holds for the odd case.

\Section{Bump-Friedberg integral}\label{sec:BF}
 Bump and Friedberg introduced
another kind of Rankin-Selberg type zeta integrals
related to exterior square $L$-functions
in \cite{B-F}.
In this section,
we treat Bump-Friedberg integrals.
Set $m =\lfloor (n+1)/2\rfloor$
and $m' = \lfloor n/2\rfloor$.
We define an embedding $J: G_m \times G_{m'}
\rightarrow G_n$
by
\[
J(g, g')_{k, l}
=
\left\{
\begin{array}{cl}
g_{ij}, & \mbox{if}\ k = 2i,\ l = 2j\\
g'_{ij}, & \mbox{if}\ k = 2i-1,\ l = 2j-1\\
0, & \mbox{otherwise},
\end{array}
\right.
\]
for $n$ even,
and by
\[
J(g, g')_{k, l}
=
\left\{
\begin{array}{cl}
g_{ij}, & \mbox{if}\ k = 2i-1,\ l = 2j-1\\
g'_{ij}, & \mbox{if}\ k = 2i,\ l = 2j\\
0, & \mbox{otherwise},
\end{array}
\right.
\]
for $n$ odd.

Let $\pi$ be an irreducible generic representation of $G_n$.
For $W \in \mathcal{W}(\pi, \psi)$
and $\Phi \in \Sch{F^m}$,
we define
\begin{eqnarray*}
Z(s_1, s_2, W, \Phi)
& = &
\int_{U_{m'}\backslash G_{m'}}
\int_{U_m\backslash G_m}
W(J(g, g')) \Phi(e_mg) |\det g|^{1/2+s_2-s_1} |\det g'|^{s_1-1/2}
dg dg'
\end{eqnarray*}
if $n$ is even,
and 
\begin{eqnarray*}
Z(s_1, s_2, W, \Phi)
& = &
\int_{U_{m'}\backslash G_{m'}}
\int_{U_m\backslash G_m}
W(J(g, g')) \Phi(e_mg) |\det g|^{s_1} |\det g'|^{s_2-s_1}
dg dg'
\end{eqnarray*}
if $n$ is odd.

We shall show that 
Bump-Friedberg integral attains the formal exterior square $L$-function
when the Whittaker function is associated to a newform.
\begin{thm}\label{thm:B-F}
Let $\pi$ be an irreducible generic representation of $\mathrm{GL}_{n}(F)$
and let
$W$ be a non-zero newform in $\mathcal{W}(\pi, \psi)$.
Then 
$Z(s_1, s_2, W, \Phi_{c(\pi)})$ is equal to $L(s_1, \pi)\fLl$
up to constant,
where 
$\Phi_{c(\pi)}$ is the characteristic function of
$\mi^{c(\pi)} \oplus \cdots \oplus \mi^{c(\pi)} \oplus 
(1+\mi^{c(\pi)}) \subset F^m$.
\end{thm}

\begin{proof}
If $c(\pi)$ is zero,
then the theorem follows from \cite{B-F} Theorem 3
with easy modification.
We assume that $c = c(\pi)$ 
is positive.
Suppose that $n$ is even.
Then we have $m = m' = n/2$.
For $g, g' \in G_m$,
let
\[
g = uak,\ g' = u' a' k'
\]
be their Iwasawa decompositions in $G_m = U_m T_m K_{m, 0}$.
By Lemma~\ref{lem:decomp}, the function
$g \mapsto \Phi_c(e_m g)$ is the characteristic function
of $U_m T_{m, 1}K_{m, c}$.
Thus,
under a suitable choice of Haar measures,
 we obtain
\begin{eqnarray*}
&  & Z(s_1, s_2, W, \Phi_c)\\
& = &
\int_{K_{m, 0}}
\int_{T_{m}}
\int_{K_{m, c}}
\int_{T_{m,1}}
W(J(ak, a'k')) |\det a|^{1/2+s_2-s_1} |\det a'|^{s_1-1/2}
\delta_{B_m}(a)^{-1}
\delta_{B_m}(a')^{-1}
da dk da' dk'.
\end{eqnarray*}
We have
$J(ak, a'k') = J(a, a') J(k, k')$
and $J(k, k') \in K_{n, c}$
since $k$ belongs to $K_{m, c}$.
Observe that $\delta_{B_m}(a)
\delta_{B_m}(a')
 |\det a|^{-1/2} |\det a'|^{1/2}
= \delta_{B_n}(J(a, a'))^{1/2}$.
Hence we obtain
\begin{eqnarray*}
 Z(s_1, s_2, W, \Phi_c)
& = &
\int_{T_{m}}
\int_{T_{m,1}}
W(J(a, a')) |\det a|^{s_2-s_1} |\det a'|^{s_1}
\delta_{B_n}(J(a, a'))^{-1/2}
da da'.
\end{eqnarray*}
Here we note  that $W$ is fixed by $K_{n, c}$.
By \cite{M5} Theorem 4.1,
we have
\begin{eqnarray*}
 Z(s_1, s_2, W, \Phi_c)
 =
 \sum s_{(f_1, f_2, \ldots, f_{n-1}, 0)}(\alpha)
 q^{-s_2\sum_{i = 1}^{m-1}f_{2i}}
 q^{-s_1\sum_{i=1}^{n-1}(-1)^{i+1}f_i},
\end{eqnarray*}
where the summation is over all
$(f_1, f_2, \ldots, f_{n-1}) \in \Z^{n-1}$
such that $f_1 \geq f_2 \geq \ldots \geq f_{n-1} \geq 0$.
Due to \cite{B-F} (3.3),
we get
\begin{eqnarray*}
 Z(s_1, s_2, W, \Phi_c)
& = &
(1-\omega q^{-ms_2})L(s_1, \pi)\fLl,
\end{eqnarray*}
where $\omega = \alpha_1 \cdots \alpha_n$ (for the definition of $\alpha_i$, see (\ref{eq:L})).
Since we are assuming that $c(\pi)$ is positive,
we have $\omega = 0$,
so that 
\begin{eqnarray*}
 Z(s_1, s_2, W, \Phi_c)
& = &
L(s_1, \pi)\fLl,
\end{eqnarray*}
as required.

We can prove the theorem for the odd case
in a similar fashion.
So the proof is complete.
\end{proof}

\section{The Galois side via the local Langlands correspondence}\label{sec:append}
In previous sections, we have defined $\mathscr{L}(s,\pi, \wedge^2)$ and shown that 
it divides $L_{{JS}}(s,\pi,\wedge^2)$. 
In this section, we collect the facts corresponding these in the Galois side 
via the local Langlands correspondence (say LLC for short). 

Let $\Omega$ be an algebraically closed field of characteristic zero. 
Let $F$ be a finite extension of $\Q_p$ and $\F_q$ be its residue field with the cardinality $q$.  
Define the inertia group $I_F$ by the following exact sequence: 
$$1\lra I_F\lra {\rm Gal}(\overline{F}/F)\stackrel{\iota}{\lra} {\rm Gal}(\overline{\F}_q/\F_q) \lra 1.$$
Take the geometric Frobenius element ${\rm Frob}_q\in {\rm Gal}(\overline{\F}_q/\F_q)\simeq \widehat{\Z}$. 
Then the Weil group $W_F$ is defined by 
the inverse image of $\Z$-span ${\rm Frob}^{\Z}_q$ by $\iota$, 
hence we have 
$$1\lra I_F\lra W_F\stackrel{\iota}{\lra} 
{\rm Frob}^\Z_q\simeq \Z \lra 1.$$ 
If we fix a lift $\Phi$ of ${\rm Frob}_q$, then 
$W_F$ can be written as $W_F=\coprod_{n\in \Z}\Phi^n I_F$. 
 
A Weil-Deligne representation of $W_F$ is a couple of a smooth representation 
$r=(r,V)$ of $W_F$ with a finite dimensional vector space $V$ over $\Omega$ and   
$N\in {\rm End}_{\Omega}(V)$ satisfying the following relation:
if $g=\Phi^n\sigma, n\in\Z, \sigma\in I_F$, then
$$r(g)Nr(g)^{-1}=q^{-n}N.$$

Let $(r,N)$ be a Weil-Deligne representation of $W_F$. 
By Jordan decomposition, $r(\Phi)$ can be written as the product of 
a semisimple matrix $T$ and a unipotent matrix $U$. Then for $g=\Phi^n\sigma, n\in\Z, \sigma\in I_F$ we define 
$$r^{{\rm ss}}(g):=T^n r(\sigma).$$
We call $r^{{\rm ss}}$ the $\Phi$-semisimplification of $r$. 
It is easy to see that a couple $(r^{{\rm ss}},N)$ forms a Weil-Deligne representation. 
We say that $r$ is  $\Phi$-semisimple if 
$r^{{\rm ss}}=r$. 
For any  Weil-Deligne representation $(r,N)$, we define the $L$-function of it by 
$$L(s,r):={\rm det}(1-q^{-s}r(\Phi)|({\rm Ker}N)^{I_F})^{-1}.$$

For each integer $n\ge 1$, 
denote by $\mathscr{G}_F(n)$ the set of isomorphism classes of  $\Phi$-semisimple 
Weil-Deligne representations of dimension $n$ and $\mathscr{A}_F(n)$ the set of isomorphism classes of smooth 
irreducible representations of GL$_n(F)$. Then by \cite{h&t},\cite{hen1}, there exists a canonical 
bijective correspondence $\mathscr{G}_F(n)\stackrel{\tiny{{\rm LLC}}}{\lra} \mathscr{A}_F(n),\ \rho\mapsto \pi(\rho)$ which is 
preserving $L$-functions and $\varepsilon$-factors of both sides (see \cite{hen2}, \cite{hen3}). 

Let $\pi$ be an irreducible generic representation of GL$_n(F)$ and $\rho$ be the corresponding 
Weil-Deligne representation via LLC.  
Recall that by using ${L}(s,\pi)=\ds\prod_{i=1}^n (1-\alpha_i q^{-s})^{-1}$, we define  
$$\mathscr{L}(s,\pi,\wedge^2)=\prod_{1\le i<j\le n}(1-\alpha_i\alpha_j q^{-s})^{-1}.$$

\begin{thm}\label{thm:galois}
The following properties are satisfied:

\medskip
$($i$)$ $\mathscr{L}(s,\pi,\wedge^2)=
{\rm det}(1-q^{-s}\rho(\Phi)|\wedge^2({\rm Ker}N)^{I_F})^{-1}$,

\medskip
$($ii$)$ $\mathscr{L}(s,\pi,\wedge^2)$ divides $L(s,\wedge^2\rho)$.
\end{thm}
\begin{proof}By definition, (i) is easy to follow. 
Note that $\wedge^2\rho=(\wedge^2\rho, N\otimes 1+1\otimes N)$. 
Then we see that $\wedge^2 {\rm Ker}N\subset {\rm Ker}(N\otimes 1+1\otimes N)$ and hence we have 
$$\wedge^2 ({\rm Ker}N)^{I_F}\subset (\wedge^2 {\rm Ker}N)^{I_F}\subset {\rm Ker}(N\otimes 1+1\otimes N)^{I_F}.$$
These inclusions are stable under the action of $\Phi$, hence the claim of (ii). 
\end{proof}

In what follows we study when our $\mathscr{L}(s,\pi, \wedge^2)$ coincides with $L(s,\pi,\wedge^2)$. 
To do this, we first prove the following:
\begin{lem}\label{psr}Let $\pi={\rm Ind}^{G_n}_{B_n}(\chi_1\otimes\cdots\otimes\chi_n)$ be a 
principal series representation. 
Then 
$$L(s,\pi,\wedge^2)=\prod_{1\le i<j\le n}L(s,\chi_i\chi_j).$$
\end{lem}
\begin{proof}
Let $\rho=(\rho,N)$ be the corresponding 
Weil-Deligne representation via LLC. Then $N=0$ and $\rho:W_F
\twoheadrightarrow  W^{{\rm ab}}_F\simeq F^{\times}
\stackrel{\oplus_{i=1}^n \chi_i }{\lra} {\rm GL}_n(\Omega)$ (cf. Theorem 4.2.1 of \cite{kudla}), hence we have 
$\wedge^2\rho=\oplus_{1\le i<j\le n}\Omega(\chi_i\chi_j)$. 
Then we have that 
$$L(s,\pi,\wedge^2)=L(s,\wedge^2\rho)=\prod_{1\le i<j\le n}L(s,\chi_i\chi_j).$$
\end{proof}

For characters $\chi_1,\ldots,\chi_n$ of $F^\times$ we settle the following hypothesis:
$$\mbox{(H): for each $i,j\ (1\le i< j\le n)$, if $\chi_i,\chi_j$ are ramified, then 
so is $\chi_i\chi_j$}.$$ 

\begin{prop}\label{prop:H}
Keep the notations in Lemma~\ref{psr}.
Under the hypothesis $(H)$, the following equality holds:
$$\mathscr{L}(s,\pi,\wedge^2)=L(s,\pi,\wedge^2).$$
\end{prop}
\begin{proof}By the hypothesis (H), the definition of $\mathscr{L}(s,\pi,\wedge^2)$ and Lemma \ref{psr}, 
we have that 
$$\mathscr{L}(s,\pi,\wedge^2)=\prod_{1\le i<j\le n}L(s,\chi_i\chi_j)=L(s,\pi,\wedge^2).$$ 
\end{proof}

\begin{rem}
Let $\pi={\rm Ind}^{G_n}_{B_n}(\chi_1\otimes\cdots\otimes\chi_n)$ be a unitary
principal series representation.
Suppose that the characters 
$\chi_1,\ldots,\chi_n$ satisfy the hypothesis (H).
Then Theorems~\ref{thm:main_even}, \ref{thm:main_odd} and Proposition~\ref{prop:H}
says that
Jacquet-Shalika integral of a newform for $\pi$ attains
the exterior square $L$-function $L(s,\pi,\wedge^2)$.
A similar result holds for Bump-Friedberg integral 
because of Theorem~\ref{thm:B-F}.
\end{rem}

\providecommand{\bysame}{\leavevmode\hbox to3em{\hrulefill}\thinspace}
\providecommand{\MR}{\relax\ifhmode\unskip\space\fi MR }
\providecommand{\MRhref}[2]{%
  \href{http://www.ams.org/mathscinet-getitem?mr=#1}{#2}
}
\providecommand{\href}[2]{#2}

\end{document}